\documentclass{amsart}

\usepackage{amssymb}
\usepackage{graphicx}
\usepackage[cmtip,all]{xy}

\newtheorem{theorem}{Theorem}[section]
\newtheorem{lemma}[theorem]{Lemma}
\newtheorem{cor}[theorem]{Corollary}

\theoremstyle{definition}

\theoremstyle{remark}
\newtheorem{remark}[theorem]{Remark}

\numberwithin{equation}{section}

\newcommand{\abs}[1]{\left\lvert#1\right\rvert}
\newcommand{\ds}{\displaystyle}

\newcommand{\CA}{\mathcal{A}}
\newcommand{\CB}{\mathcal{B}}
\newcommand{\CC}{\mathcal{C}}
\newcommand{\CH}{\mathcal{H}}
\newcommand{\CI}{\mathcal{I}}
\newcommand{\CJ}{\mathcal{J}}
\newcommand{\CN}{\mathcal{N}}

\begin{document}

\title[Poisson phenomena for points on hyperelliptic curves]
{Poisson type phenomena for points on hyperelliptic curves modulo $p$}

\author{Kit-Ho Mak}
\address{Department of Mathematics \\
University of Illinois at Urbana-Champaign \\
273 Altgeld Hall, MC-382 \\
1409 W. Green Street \\
Urbana, Illinois 61801, USA}
\email{mak4@illinois.edu}

\author{Alexandru Zaharescu}
\address{Department of Mathematics \\
University of Illinois at Urbana-Champaign \\
273 Altgeld Hall, MC-382 \\
1409 W. Green Street \\
Urbana, Illinois 61801, USA}
\email{zaharesc@math.uiuc.edu}


\thanks{The second author is supported by NSF grant number DMS - 0901621.}

\begin{abstract}
Let $p$ be a large prime,  and let $C$ be a hyperelliptic curve over $\mathbb{F}_p$. We study the distribution of the $x$-coordinates in short intervals when the $y$-coordinates lie in a prescribed interval, and the distribution of the distance between consecutive $x$-coordinates with the same property.  Next, let $g(P,P_0)$ be a rational function of two points on $C$. We  study the distribution of the above distances with an extra condition that $g(P_i,P_{i+1})$ lies in a prescribed interval, for any consecutive points $P_i,P_{i+1}$.
\end{abstract}

\maketitle

\section{Introduction}

Let $p$ be a large prime. In \cite{CoZa98}, Cobeli and one of the authors considered the distribution of $r$-tuples of primitive roots modulo $p$. They showed that the distribution of primitive roots becomes Poissonian
as $p$ tends to infinity via a sequence of primes
such that $\varphi(p-1)/p\rightarrow 0$. Moreover, they showed that the proportion of distances between consecutive primitive roots which are at least $\lambda$ times larger than the average value $p/\varphi(p-1)$ tends to $e^{-\lambda}$. In this paper, we employ an analogous technique to study $r$-tuples of $x$-coordinates on a hyperelliptic curve modulo a large prime number $p$.

Let $C$ be a hyperelliptic curve over $\mathbb{F}_p$ defined by the equation $y^2=f(x)$, $f$ not a square. Let $\CI$ be an interval inside $[0,(p-1)/2]$ with $\abs{\CI} \geq p/\log\log{p}$, $\abs{\CI}=o(p)$. We consider the $x$-coordinates of the points $(x,y)\in C$ with $y\in\CI$, and denote them  $0\leq x_1<\dots<x_m\leq p-1$.
We study the distribution of the number of such $x_i$'s in $(x,x+t]$, where $x$ itself is one of such $x_i$'s, and $t\sim\lambda p/\abs{\CI}$. It turns out that under certain natural assumptions, as $p$ increases, the distribution approaches the Poisson distribution with parameter $\lambda$.

Next, we consider the proportion of distances between consecutive $x_i$'s which are at least $\lambda$ times greater than the asymptotic average $p/\abs{\CI}$, that is,
\begin{equation*}
\mu(\lambda) = \frac{\#\{ i:1\leq i \leq m, x_{i+1}-x_i \geq \lambda p/\abs{\CI} \}}{m}
\end{equation*}
where $m$ is the total number of such $x_i$, and $x_{m+1}=x_1+p$. As $p$ tends to infinity, we show that the limit of $\mu_p(\lambda)$ tends to $e^{-\lambda}$, and moreover, that this convergence is uniform on compact subsets of $[0,\infty)$.

Lastly, we go a step further and investigate to what extent the above Poisson distribution might be distorted via a rational function $g(P,P_0)$ of two points $P,P_0$ on the curve $C$. This builds on, and extends some ideas from \cite{VaZa02}.
More precisely, we study the distribution of the number of $x_i$'s in $(x,x+t]$ as above, but with the extra condition that $g(P_{i},P_{i+1})\in\CJ$, where $\CJ=[\alpha p,\beta p)$, and $P_i=(x_i,y_i), P_{i+1}=(x_{i+1},y_{i+1})$ are points on $C$ with $y_i,y_{i+1}\in\CI$. The resulting distribution is again Poisson, but with a different parameter $\lambda'=\lambda(\beta-\alpha)$. Regarding the proportion of distances between consecutive $x_i$'s satisfying the above extra condition, that is,
\begin{equation}\label{defm}
\mu(\lambda,\alpha,\beta)=\frac{\#\{ i:1\leq i \leq m, x_{i+1}-x_i \geq \lambda p/\abs{\CI},g(P_{i},P_{i+1})\in [\alpha p,\beta p] \}}{m},
\end{equation}
we show that as $p$ tends to infinity, $\mu_p(\lambda,\alpha,\beta)$ tends to $e^{-\lambda(\beta-\alpha)}$.

As an application of our results, we will derive a  result which shows how the distribution of distances between the $x$-coordinates of points on an elliptic curve $C$ is affected
by the group law of $C$. This is our original motivation for studying this problem.

\section{Distribution of Values of Rational Maps on an affine curve in a hypercube modulo $p$}

Since the Poisson distribution of $x$-coordinates in short intervals without any distortion $g$ and its corresponding limit distribution of consecutive difference can be derived from the case with distortion by simply setting $\CJ=[0,p)$, i.e. $\alpha=0, \beta=1$, we will proceed directly to prove the results when distortion exists, and derive the case without distortion as a corollary.

Let $p$ be a large prime number, and let $X$ be an irreducible affine curve over $\mathbb{A}^r_p$, the affine $r$-space over $\mathbb{F}_p$, given by the set of equations $f_i(\boldsymbol{x})=0$, where $\boldsymbol{x}=(x_1,\dots,x_r)$, $1\leq i \leq k$. By the well known Weil bounds for space curves \cite{AuPe96}  (note that in our case $X$ is affine instead of projective) we know that

\begin{equation}\label{weilboundforX}
\abs{ \#X-p } \leq 2g_a \sqrt{p},
\end{equation}
where $g_a$ denotes the arithmetic genus of $X$. Note that this formula works even when $X$ is singular.

Let $\boldsymbol{g}=(g_1,\dots,g_s)$ be a rational map from $X$ to $\mathbb{A}^s_p$. Thus each $g_i$ is a quotient of polynomials in $\mathbb{F}_p[x_1,\dots,x_r]$. Recall that the degree $\deg{(g_i)}$ of $g_i$ is defined as the maximum between the degree of its numerator and the degree of its denominator. Define the degree of the rational map $\boldsymbol{g}$ to be $\deg{(\boldsymbol{g})} := \max_{1\leq i\leq s}\deg{(g_i)}$.

For the convenience of the reader, we recall the notion of linear independence on a curve $X$.
A set of functions $\{g_1, \dots, g_s\}$ is linearly independent provided that
if $c_1,\dots,c_s\in\overline{\mathbb{F}_p}$ are such that $c_1g_1(\boldsymbol{x})+\dots+c_sg_s(\boldsymbol{x})=0$ on the curve $X$, then $c_1=\dots=c_s=0$.

Let $\CI_1, \dots, \CI_r$ be intervals in $[0,p)$, and we view $\CI_1\times\dots\times\CI_r \subset \mathbb{A}^r$ as a hypercube in the domain for which $X$ is defined. Similarly, let $\CJ_1,\dots, \CJ_s$ be intervals in $[0,p)$, and view $\CJ_1\times\dots\times\CJ_s \subset \mathbb{A}^s$ as a hypercube in the range of the rational map $\boldsymbol{g}=(g_1,\dots,g_s)$. We define
\begin{equation*}
\CN(X) = \#\{\boldsymbol{x}\in X\cap(\CI_1 \times \dots \times \CI_r)|\boldsymbol{x} \text{~is not a pole of~} \boldsymbol{g}, \boldsymbol{g(x)}\in \CJ_1\times\dots\times\CJ_s \}
\end{equation*}
to be the number of points on $X$ lying inside the hypercube $\CI_1 \times \dots \times \CI_r$, whose images under $\boldsymbol{g}$ lie inside the hypercube $\CJ_1\times\dots\times\CJ_s$. The main result of this section is the following theorem, which may be regarded as a uniform distribution result, where the intervals $\CI_i, \CJ_j$ are not too small.

\begin{theorem}\label{ptsincube}
Let $X$ be as above, of degree $d>1$, and let $\abs{\CI}$ denote the number of integers inside the interval $\CI$. Let $\boldsymbol{g}=(g_1,\dots,g_s)$ be a rational map with $1\leq\deg{\boldsymbol{g}}<d$. Let $p$ be a large prime, and assume that the set of functions $\{1, x_1, \dots, x_r, g_1(\boldsymbol{x}), \dots, g_s(\boldsymbol{x})\}$ is linearly independent on $X$. Then
\begin{equation*}
\abs{ \CN(X) - \frac{\abs{\CI_1}\dots\abs{\CI_r}\abs{\CJ_1}\dots\abs{\CJ_s}}{p^{r+s-1}} } \leq 2^{r+s} d(d-1)\sqrt{p}\log^{r+s}{p} + O(\sqrt{p}\log^{r+s-1}{p}).
\end{equation*}
\end{theorem}

\begin{remark}
The uniform distribution problem of rational points over an irreducible variety in a hypercube was investigated by Myerson \cite{Mye81}, and also Fujiwara \cite{Fuj88} (see also \cite{CoZa01} for the case of curves, but with more general regions). Their results were improved in the case for complete intersections by Shparlinski and Skorobogatov \cite{ShSk90}, Skorobogatov \cite{Sko92} and Luo \cite{Luo99}. On the other hand, the uniform distribution problem of rational maps was investigated by Vajaitu and one of the authors \cite{VaZa02} (see also \cite{Zah03} and \cite{GSZ05} for other related distribution problems). Here, we combine both ideas, and at the same time produce an explicit error term for later use.
\end{remark}

The first step in the proof of Theorem \ref{ptsincube} is to rewrite $\CN(X)$ as an exponential sum.

\begin{lemma}\label{lemmaNX}
Denote $e_p(y)=e^{2\pi i y/p}$, and let $T=\{(t_1,\dots,t_r): \abs{t_j}\leq (p-1)/2\}$ and $U=\{(u_1,\dots,u_s): \abs{u_j}\leq (p-1)/2\}$. We have
\begin{align*}
\CN(X) &= \frac{1}{p^{r+s}}\sum_{(t_1,\dots,t_r) \in T} \prod_{1\leq i \leq r} \left( \sum_{m_i\in\CI_i}e_p(t_i m_i) \right) \\
&\qquad\times\sum_{(u_1,\dots,u_s) \in U}\prod_{1\leq j \leq s} \left( \sum_{n_j\in\CJ_j}e_p(u_j n_j) \right) \\
&\qquad\qquad\times\sideset{}{'}\sum_{\substack{\boldsymbol{x}\in X \\ 0\leq x_i\leq p-1}} e_p(-u_1g_1(\boldsymbol{x})-\dots-u_sg_s(\boldsymbol{x})-t_1x_1-\dots -t_rx_r),
\end{align*}
where $\sideset{}{'}\sum$ means we ignore the poles of the $g_i$'s when summing.
\end{lemma}
\begin{proof}
From the orthogonal relation of the exponential sum
\begin{equation*}
\sum_{\abs{t_i}\leq (p-1)/2} e_p(t_i(m_i-x_i)) =
\begin{cases}
p & \text{if}~ x_i=m_i, \\
0 & \text{if}~ x_i\neq m_i,
\end{cases}
\end{equation*}
we sum over all possible $m_i\in\CI_i$ to get
\begin{equation*}
\frac{1}{p}\sum_{m_i\in\CI_i}\sum_{\abs{t_i}\leq (p-1)/2} e_p(t_i(m_i-x_i)) =
\begin{cases}
1 & \text{if}~ x_i\in\CI_i, \\
0 & \text{if}~ x_i\notin\CI_i.
\end{cases}
\end{equation*}
In the same spirit, we get
\begin{equation*}
\frac{1}{p}\sum_{n_j\in\CJ_j}\sum_{\abs{u_j}\leq (p-1)/2} e_p(u_j(n_j-g_j(\boldsymbol{x})) =
\begin{cases}
1 & \text{if}~ g_j(\boldsymbol{x})\in\CJ_j, \\
0 & \text{if}~ g_j(\boldsymbol{x})\notin\CJ_j.
\end{cases}
\end{equation*}
Therefore,
\begin{multline*}
\frac{1}{p^{r+s}}\prod_{1\leq i \leq r}\prod_{1\leq j \leq s} \sum_{m_i\in\CI_i}\sum_{\abs{t_i}\leq (p-1)/2}\sum_{n_j\in\CJ_j}\sum_{\abs{u_j}\leq (p-1)/2} e_p(t_i(m_i-x_i)) e_p(u_j(n_j-h_j(\boldsymbol{x})) \\
=\begin{cases}
1 & \text{if}~ \boldsymbol{x}\in\CI_1\times\dots\times\CI_r \text{~and}~ \boldsymbol{g(x)}\in\CJ_1\times\dots\times\CJ_s, \\
0 & \text{otherwise}.
\end{cases}
\end{multline*}
Finally, $\CN(X)$ is the sum of the above quantity over all possible $\boldsymbol{x}\in X$. By rearranging terms on the repeated sums we get the lemma.

\end{proof}

The main term of $\CN(X)$ corresponds to the term with all $t_i=u_j=0$ in the above lemma, which is
\[\frac{1}{p^{r+s}}\abs{\CI_1}\dots\abs{\CI_r}\abs{\CJ_1}\dots\abs{\CJ_s}\#X(\mathbb{F}_p).\]
By (\ref{weilboundforX}), this is
\begin{align}
\text{main term~} &= \frac{1}{p^{r+s}}\abs{\CI_1}\dots\abs{\CI_r}\abs{\CJ_1}\dots\abs{\CJ_s}(p+O(\sqrt{p})) \nonumber \\
&= \frac{1}{p^{r+s-1}}\abs{\CI_1}\dots\abs{\CI_r}\abs{\CJ_1}\dots\abs{\CJ_s}+O((r+s)\sqrt{p}). \label{maintermNX}
\end{align}

The following two lemmas estimate the remaining terms.
\begin{lemma}\label{est1}
Let $p$ be a large prime. For any interval $\CI$, we have
\begin{equation*}
\abs{\sum_{1\leq \abs{t} \leq \frac{p-1}{2}} \sum_{m\in\CI}e_p(tm)} \leq 2p\log{p}.
\end{equation*}
\end{lemma}
\begin{proof}
Let $\CI\cap\mathbb{Z} = \{l,l+1,\dots,l+h-1\}$, where $h=\abs{\CI}$. Then
\begin{equation*}
\sum_{m\in\CI}e_p(tm) =
\begin{cases}
h & \text{if}~ t=0, \\
\left(e^{\frac{-2\pi itl}{p}}\right)\frac{1-e^{-2\pi ith/p}}{1-e^{-2\pi it/p}} & \text{if}~ t\neq 0.
\end{cases}
\end{equation*}
Hence if $t\neq 0$,
\begin{equation*}
\abs{\sum_{m\in\CI}e_p(tm)} \leq \frac{2}{\abs{1-e^{-2\pi it/p}}}.
\end{equation*}
Since $\abs{1-e^{-2\pi it/p}} = 2\abs{\sin{(\pi t/p)}} \geq \frac{\pi\abs{t}}{p}$ for $p$ large enough, we obtain the estimate
\begin{equation*}
\abs{\sum_{m\in\CI}e_p(tm)} \leq \frac{2p}{\pi\abs{t}} \leq \frac{p}{\abs{t}}.
\end{equation*}
Finally, the lemma is obtained by summing over all $t$ with $1\leq\abs{t}\leq (p-1)/2$, using the elementary inequality
\begin{equation*}
1+\frac{1}{2}+\dots+\frac{1}{\frac{p-1}{2}} \leq \log{p}.
\end{equation*}
\end{proof}

\begin{lemma}\label{estexp}
If $(t_1,\dots,t_r,u_1,\dots,u_s)\neq(0,\dots,0)$, the degree $d$ of $X$ is greater than $1$, $1\leq\deg{(\boldsymbol{g})}<d$, and the set $\{1, x_1, \dots, x_r, g_1(\boldsymbol{x}), \dots, g_s(\boldsymbol{x})\}$ is linearly independent on $X$, then
\begin{equation*}
\abs{\sum_{\substack{\boldsymbol{x}\in X \\ 0\leq x_i\leq p-1}} e_p(-u_1g_1(\boldsymbol{x})-\dots-u_sg_s(\boldsymbol{x})-t_1x_1-\dots -t_rx_r)} \leq d(d-1)\sqrt{p}+\frac{1}{2}d^2.
\end{equation*}
\end{lemma}
\begin{proof}
Apply Theorem 6 of \cite{Bom66} to the projective closure of $X$. (Note that since $1\leq\deg{(\boldsymbol{g})}<d$, the assumption of that theorem is satisfied for all large enough $p$.)
\end{proof}

\begin{proof}[Proof of Theorem \ref{ptsincube}]
With the main term (\ref{maintermNX}) already established, we only have to estimate other terms corresponding to nonzero $(t_1,\dots,t_r,u_1,\dots,u_s)$ in Lemma \ref{lemmaNX}. The innermost sum for those terms is estimated uniformly by Lemma \ref{estexp}. Now group those terms according to the number of nonzero $t_i$ and $u_j$, use Lemma \ref{est1} for nonzero $t_i, u_j$, and the trivial estimate $\sum_{m_i\in\CI_i}e_p(t_i m_i) \leq p$ for $t_i=0$ (or the equivalent for
 $u_j=0$). We see that the absolute value of the remaining terms is less than or equal to
\begin{multline*}
(2^{r+s}\log^{r+s}{p}+(r+s)2^{r+s-1}\log^{r+s-1}{p}+\binom{r+s}{2}2^{r+s-2}\log^{r+s-2}{p}+\dots\\
+2(r+s)\log{p}) \times(d(d-1)\sqrt{p}+\frac{1}{2}d^2)
\end{multline*}
which is
\[2^{r+s} d(d-1)\sqrt{p}\log^{r+s}{p} + O(\sqrt{p}\log^{r+s-1}{p}).\]
This finishes the proof of Theorem \ref{ptsincube}.
\end{proof}

\begin{remark}\label{remimp}
From the proof of Theorem \ref{ptsincube}, we see that if some of the intervals among $\CI_i, \CJ_j$ are the full interval $[0,p)$, then we can loosen the linearly independent condition a little bit. This will be vital in our application later.

Let $\CI_i$ correspond to the coordinate functions $x_i$, and $\CJ_j$ correspond to the functions $g_j(\boldsymbol{x})$. From the proof of Lemma \ref{lemmaNX}, we see that if any of the $\CI_i$ or $\CJ_j$ is the full interval, the exponential sum over that interval and its corresponding function can be omitted. Thus when we apply Bombieri's estimate in Lemma \ref{estexp}, we can remove the function from the set we require to be linearly independent if its corresponding interval is the full one.
\end{remark}

As an example of how we make use of the above remark, we let $r=s$ and $\boldsymbol{g}$ to be the identity map. Then it is not necessary to restrict our range to any subset. Hence all $\CJ_j$'s are full. From the remark we only need to ensure the set $\{1, x_1, \dots, x_r\}$ is linearly independent, and this is true since the degree of $X$ is $d>1$. Thus we recover the uniform distribution theorem in \cite{Fuj88}, now with an explicit error term.

\begin{cor}\label{ptsincube2}
Let $X$ be as usual, of degree $d>1$, and $\abs{\CI}$ denotes the number of integers inside $\CI$. Let
\begin{equation*}
\CN'(X) = \#\{\boldsymbol{x}\in X\cap(\CI_1 \times \dots \times \CI_n) \}
\end{equation*}
the number of points of $X$ lying inside the hypercube $\CI_1 \times \dots \times \CI_n$. If $p$ is a large prime, then
\begin{equation*}
\abs{ \CN'(X) - \frac{\abs{\CI_1}\dots\abs{\CI_r}}{p^{r-1}} } \leq 2^r d(d-1)\sqrt{p}\log^r{p} + O(\sqrt{p}\log^{r-1}{p}).
\end{equation*}
\end{cor}

As another application, if we do not restrict our domain, that is if all $\CI_i$'s are full, then using Remark \ref{remimp} we recover \cite[Theorem 1]{VaZa02} in the special case where $\Omega$ is a hypercube.

\section{$r$-tuples of $x$-coordinates of a hyperelliptic curve mod $p$ with a prescribed rational function}

For the remaining of the paper, we let $C$ be a hyperelliptic curve over $\mathbb{F}_p$ defined by the equation $y^2=f(x)$, with $d=\deg{f}$. We assume $f$ is not a square in $\overline{\mathbb{F}_p}(x)$, so that $C$ is irreducible. We are interested in the distribution of the distances between successive $x$-coordinates of
points on $C$, subjected to a restricted range of a rational function in terms of the two successive points (we will make this precise in a moment). Our approach is inspired by \cite{CoZa98}, where the distribution of the distances between successive primitive roots mod $p$ is studied.

Let $\CH=\{h_1,\dots,h_r\}$ be a subset of $\{1,2,\dots,p-1\}$. To each pair of $(C,\CH)$, we define the \textit{$x$-shifted curve} of $C$ by $\CH$, $C_{\CH}$, to be the curve defined by the family of equations
\begin{align*}
y^2 &= f(x) \\
y_1^2 &= f(x+h_1) \\
&\vdots \\
y_r^2 &= f(x+h_r)
\end{align*}
in $\mathbb{A}^{r+2}_p$ with the $r+2$ coordinates $x,y,y_1,\dots,y_r$ (the shifted curve also appeared in \cite{MaZa11}, but the definition here is a little bit different). It is easy to see that $C_{\CH}$ is indeed a curve.

Let $S$ be the set of all $x\in\mathbb{F}_p$ so that there is a $y$ with $(x,y)\in C(\mathbb{F}_p)$. From the definition of $C_{\CH}$, it is obvious that a point on $C_{\CH}$ corresponds to an $x$ such that $x$ and $x+h_i$ are all in $S$ for all $h_i\in\CH$.

More generally, if $\CI$ is an interval in $[0,p)$, let $S_{\CI}$ be the set of all $x$ so that there is a $y\in\CI$ with $(x,y)\in C(\mathbb{F}_p)$. Then there is a correspondence from the set of points on $C_{\CH}$ inside the hypercube $([0,p)\times\CI^{r+1})$ to the set of $x$'s so that all $x$ and $x+h_i$ are in $S_{\CI}$ for all $h_i\in\CH$.

Now suppose $P=(x,y)$ and $P_0=(x_0,y_0)$ are two points on $C$, $g=g(P,P_0)=g(x,y,x_0,y_0)$ is a rational function between the $2$ points. With respect to a point $P=(x,y)$, we define $S_{\CI,\CJ,P}$ to be the set of all $x_0$ in $S_{\CI}$ satisfying the extra condition $g(P,P_0)\in\CJ$, for some $P_0=(x_0,y_0)$ on $C$ with $y_0\in\CI$. If $g_i=g(P,P_i)=g(x,y,x+h_i,y_i)$ is the rational function obtained from $g$ by putting in $P_0=(x+h_i,y_i)$, and let $\boldsymbol{g}=(g_1,\dots,g_r)$, then $\boldsymbol{g}$ is a rational function on $C_{\CH}$. It is clear that there is a correspondence from the set of points on $C_{\CH}$ inside the hypercube $([0,p)\times\CI^{r+1})$ whose image under $\boldsymbol{g}$ lie in $\CJ^{r}$ to the set of $x$'s such that ($P=(x,y)$ as usual) $x+h_i$ are in $S_{\CI,\CJ,P}$ for all $h_i\in\CH$.

To simplify matters, from now on we assume that the interval $\CI\subset [0,(p-1)/2]$, so that one $x$-value can only correspond to at most one $y$-value, and hence the above correspondence is bijective. We define $\CN(\CH)=\CN(\CH;C,p,g,\CI,\CJ)$ to be the number of points on $C_{\CH}$ inside the hypercube $([0,p)\times\CI^{r+1})$ whose image under $\boldsymbol{g}$ lie in $\CJ^{r}$. The following lemma gives an idea of the size of $\CN(\CH)$.

\begin{lemma}\label{ptsinCH}
Let $C$ be a hyperelliptic curve defined by $y^2=f(x)$. Let $\CH=\{h_1,\dots,h_r\}$, $f\in\mathbb{F}_p[x]$ of degree $d$, not a square in
$\overline{\mathbb{F}_p}[x]$, and $g(P,P_0)$ a rational function between two points in $X$, of degree $\deg{g}<d$, and the set $\{1,g_1,\dots,g_r\}$ is linearly independent on $C_{\CH}$. If $d$ and $r$ are small compared to $p$, then for all sufficiently large $p$,
\begin{equation*}
\abs{ \CN(\CH) - \frac{\abs{\CI}^{r+1}\abs{\CJ}^{r}}{p^{2r}} } \leq 2^{3r+2} d(2^r d-1)\sqrt{p}\log^{2r+2}{p} + O(\sqrt{p}\log^{2r+1}{p}).
\end{equation*}
\end{lemma}
\begin{proof}
It is easy to compute that $D=\deg{C_{\CH}}=2^r d$. Once we show that $C_{\CH}$ is irreducible, this lemma will follow from replacing $r$ by $r+2$, letting $s=r$, $\CI_i=\CI$, $\CJ_j=\CJ$ and $\boldsymbol{g}=(g_1,\dots,g_r)$ (recall that $g_i=g(P,P_i)$) in Theorem \ref{ptsincube}. Note that by Remark \ref{remimp} there is no need to include the function $x$ in the set of functions we require to be linearly independent since its corresponding interval is full.

We show the irreducibility of $C_{\CH}$ by showing that the field
\[K=\mathbb{F}_p(x)\left[\sqrt{f(x)},\sqrt{f(x+h_1)},\dots,\sqrt{f(x+h_r)}\right]\]
obtained by adjoining a square root of each of the $f(x)$ and $f(x+h_i)$ is the function field of $C_{\CH}$. The condition that $r$ is small compared to $p$ ensures that $K$ is a field extension of degree $2^{r+1}$ over $\mathbb{F}_p(x)$. Now we proceed using induction on $r$.

For $r=0$ this follows from the condition that $f$ is not a square.

Assume that $K_{r-1}=\mathbb{F}_p(x)\left[\sqrt{f(x)},\sqrt{f(x+h_1)},\dots,\sqrt{f(x+h_r)}\right]$ is the function field of $C_{\CH'}$, where $\CH'=\CH-{h_r}$. $K_{r-1}$ is a field of degree $2^r$ over $\mathbb{F}_p(x)$. Let $I=\langle y^2 - f(x), y_1^2 - f(x+h_1),\dots,y_{r-1}^2 - f(x+h_{r-1})\rangle$ be the ideal of
$\CH'$. Then we have an isomorphism

\begin{equation*}
\xymatrix@1{\ds \frac{\mathbb{F}_p(x)[y,y_1,\dots,y_{r-1}]}{I} \ar[r]^-{\sim} & K_{r-1}}
\end{equation*}
obtained by sending $y\mapsto\sqrt{f(x)}, y_i\mapsto\sqrt{f(x+h_i)}$. The induction is completed if we can prove that the map
\begin{equation}\label{isofield}
\xymatrix@1{\ds \frac{K_{r-1}[y_r]}{y_r^2-f(x+h_r)} \ar[r]^-{\phi} & K=K_{r-1}\left[ \sqrt{f(x+h_r)} \right] }
\end{equation}
with $y_r \mapsto \sqrt{f(x+h_r)}$, is an isomorphism. The map $\phi$ is clearly surjective, and from the degrees of the fields $K, K_{r-1}$ over $\mathbb{F}_p(x)$, we see that $K$ is a vector space over $K_{r-1}$ of dimension $2$. Since the left hand side of (\ref{isofield}) has rank at most $2$ over $K_{r-1}$, $\phi$ must be an isomorphism.
\end{proof}

\begin{remark}
If $\CJ$ is the full interval, then by using Lemma \ref{remimp}, we can remove the assumption that the set $\{1,g_1,\dots,g_r\}$ is linearly independent on $C_{\CH}$.
\end{remark}

\begin{remark}\label{remcond}
For the rest of the paper, we will assume that as $p\rightarrow\infty$, we have $d=o(p)$, $r=o(\log{p}/\log\log{p})$, $\abs{\CI}\geq p/\log\log{p}$, and $\CJ=[\alpha p,\beta p)$ ($0\leq \alpha < \beta \leq 1$). It is clear that under these conditions, the proof of Lemma \ref{ptsinCH} works and the main term has a bigger magnitude than the error term when $p$ is sufficiently large.
\end{remark}

Next, if $\CA, \CB$ are two disjoint sets of integers, we define
\[\CN(\CA,\CB)=\CN(\CA,\CB;C,p,\CI,\CJ)\]
to be the number of $x$ such that $x$ and $x+a$ are in $S_{\CI,\CJ}$ for all $a\in\CA$, but $x+b$ are not in $S_{\CI,\CJ}$ for any $b\in\CB$. To estimate $\CN(\CA,\CB)$, we introduce the characteristic function
\begin{equation*}
\delta(x) =
\begin{cases}
1 & \text{if}~ x\in S_{\CI,\CJ}, \\
0 & \text{otherwise}.
\end{cases}
\end{equation*}
Since in our case one $x$ can correspond to at most one $y$, we can write $\CN(\CA,\CB)$ in terms of $\delta(x)$,
\begin{align*}
\CN(\CA,\CB) &= \sum_{x\in[0,p)}\prod_{a\in\CA} \delta(x+a) \prod_{b\in\CB} (1-\delta(x+b)) \\
&= \sum_{x\in[0,p)}\prod_{a\in\CA} \delta(x+a) \sum_{\CC\subset\CB}(-1)^{\abs{\CC}}\prod_{c\in\CC}\delta(x+c) \\
&= \sum_{\CC\subset\CB}(-1)^{\abs{\CC}}\sum_{x\in[0,p)}\prod_{d\in\CA\cup\CC}\delta(x+d) \\
&= \sum_{\CC\subset\CB}(-1)^{\abs{\CC}}\CN(\CA\cup\CC).
\end{align*}

Combining this with Lemma \ref{ptsinCH}, which says
\[\CN(\CH) = \frac{\abs{\CI}^{\abs{\CH}+1}\abs{\CJ}^{\abs{\CH}}}{p^{2\abs{\CH}}} + \theta_{\CH} 2^{3\abs{\CH}+2} d(2^{\abs{\CH}}d -1)\sqrt{p}\log^{2\abs{\CH}+2}{p} + O(\sqrt{p}\log^{2\abs{\CH}+1}{p}) \]
for some $\abs{\theta_{\CH}} \leq 1$, we get the following result.

\begin{theorem}\label{ptsinAB}
Let $\CA,\CB$ be two sets of integers distinct mod $p$. Then
\begin{multline*}
\abs{\CN(\CA,\CB) - \abs{\CI}\left( \frac{\abs{\CI}\abs{\CJ}}{p^2} \right)^{\abs{\CA}} \left( 1-\frac{\abs{\CI}\abs{\CJ}}{p^2} \right)^{\abs{\CB}}} \\
\leq 2^{3\abs{\CA}+4\abs{\CB}+1} d(2^{\abs{\CA}+\abs{\CB}}d -1) \sqrt{p}\log^{2\abs{\CA}+2\abs{\CB}+2}{p} + O(\sqrt{p}\log^{2\abs{\CA}+2\abs{\CB}+1}{p}).
\end{multline*}
\end{theorem}

\begin{remark}
Theorem \ref{ptsinAB} only depends on the cardinality of $\CA,\CB$ and also the number of integers in the interval $\CI$, but not the particular elements in $\CA,\CB$ and the position of $\CI,\CJ$. It is interesting to compare Theorem \ref{ptsinAB} with Theorem 1 in \cite{CoZa98}.
\end{remark}

We remind the reader that we have assumed that
\begin{align*}
d&=o(p),&\abs{\CA},\abs{\CB}&=o(\log{p}/\log\log{p})& \text{~and~}&& \abs{\CI},\abs{\CJ}&\geq p/\log\log{p}.
\end{align*}
See Remark \ref{remcond}.

\section{The Poisson distribution of the $x$-coordinates}

Recall that $S_{\CI}$ is the set of all $x$ so that there is a $y\in\CI$ with $(x,y)\in C$, and for $P=(x,y)\in C$, $S_{\CI,\CJ,P}$ is the set of all $x_0$ in $S_{\CI}$ satisfying the extra condition $g(P,P_0)\in\CJ$, for some $P_0=(x_0,y_0)$ on $C$ with $y_0\in\CI$. For $t\geq 1$ and $k$ a non-negative integer, we define $P_k(t)=P_k(t;C,p,g,\CI,\CJ)$ to be the proportion of $x\in S_{\CI}$ for which the interval $(x,x+t]$ contains exactly $k$ elements in $S_{\CI,\CJ,P}$ with $P=(x,y)$ as usual. Note that by Corollary \ref{ptsincube2}, the cardinality of $S_{\CI}$ satisfies
\begin{equation*}
\abs{ \abs{S_{\CI}} - \abs{\CI} } \leq 4 d(d-1)\sqrt{p}\log^2{p} + O(\sqrt{p}\log{p}),
\end{equation*}
or equivalently, for some $\abs{\theta} \leq 1$,
\begin{align}
\abs{S_{\CI}} &= \abs{\CI} + 4 \theta d(d-1)\sqrt{p}\log^2{p} + O(\sqrt{p}\log{p}) \label{ptsinSCI} \\
&= \abs{\CI}\left( 1+\frac{4 \theta d(d-1)\sqrt{p}\log^2{p}}{\abs{\CI}}+O(\log{p}/\abs{\CI}) \right). \nonumber
\end{align}

Next, we write $P_k(t)$ in terms of the quantities $\CN(\CA,\CB)$.
\begin{equation*}
P_k(t) = \frac{1}{\abs{S_{\CI}}}\sum_{\substack{\CC\subset\{1,\dots,[t]\} \\ \abs{\CC}=k}} \CN(\CC,\overline{\CC}),
\end{equation*}
where $\overline{\CC}=\{1,\dots,[t]\}-\CC$.

For $t=o(\log{p}/\log\log{p})$, $k\leq t$, $\abs{\CI} \geq p/\log\log{p}$, and $\CJ=[\alpha p,\beta p)$, we can apply Theorem \ref{ptsinAB} and (\ref{ptsinSCI}) to obtain
\begin{multline*}
P_k(t) = \left(1+\frac{4 \theta d(d-1)\sqrt{p}\log^2{p}}{\abs{\CI}}+O(\log{p}/\abs{\CI})\right)^{-1} \\
\times \left( \sum_{\substack{\CC\subset\{1,\dots,[t]\}\\ \abs{\CC}=k}} \left( \frac{\abs{\CI}\abs{\CJ}}{p^2} \right)^{\abs{\CC}} \left( 1-\frac{\abs{\CI}\abs{\CJ}}{p^2} \right)^{\abs{\overline{\CC}}} +E \right)
\end{multline*}
with
\[\ds E = \frac{1}{\abs{\CI}}\theta'2^{4[t]+1}d(2^{[t]}d-1)\sqrt{p}\log^{2[t]+2}{p}+O(\sqrt{p}\log^{2[t]+1}{p}), ~\abs{\theta'}\leq 1.\]
This simplifies to
\begin{equation}\label{estPKT}
P_k(t) = \binom{[t]}{k} \left( \frac{\abs{\CI}\abs{\CJ}}{p^2} \right)^k \left( 1-\frac{\abs{\CI}\abs{\CJ}}{p^2} \right)^{[t]-k} + O(p^{-1/2}\log^{2[t]+3}{p}),
\end{equation}
where the constant in $O(p^{-1/2}\log^{2[t]+3}{p})$ can be taken as $2^{4[t]+1}d(2^{[t]}d-1)$.

Suppose now $p$ goes to infinity, while $\lambda = t\abs{\CI}/p$ remains fixed (so that $t$ goes to infinity as $p\rightarrow\infty$, and automatically $\abs{\CI}=o(p)$). Note that the condition $\abs{\CI}\geq p/\log\log{p}$ guarantees that $t=O(\log\log{p})$ (so it is certainly $o(\log{p}/\log\log{p})$) and hence it guarantees that our formula works throughout the limiting process. We also have $\abs{\CJ}/p\rightarrow\beta-\alpha$. As $p\rightarrow\infty$, the error term is at most $O(p^{-1/2+\delta})$ for any $\delta>0$. Thus (\ref{estPKT}) shows that asymptotically $P_k(t)$ has a Poisson distribution with parameter $\lambda(\beta-\alpha)$:
\[\ds P_k(t) \sim e^{-\lambda(\beta-\alpha)}\frac{(\lambda(\beta-\alpha))^k}{k!} \]
for any non-negative integer $k$. More precisely, we have

\begin{theorem}\label{thmPKT}
Let $k$ be a non-negative integer. Suppose $t=O(\log\log{p})$, $\abs{\CI} \geq p(\log\log{p})^2/\log{p}$, $\abs{\CI}=o(p)$ and $\CJ=[\alpha p,\beta p)$. Set $\lambda = t\abs{\CI}/p$, then as $p$ goes to infinity, we have
\begin{equation*}
P_k(t)=e^{-\lambda(\beta-\alpha)}\frac{((\beta-\alpha)\lambda)^k}{k!}e^{O((1+k+\lambda(\beta-\alpha))\abs{\CI}/p)} \left[ 1+O\left( \frac{k^2\abs{\CI}}{\lambda p(\beta-\alpha)} \right) \right]+O(p^{\frac{-1}{2}+\delta})
\end{equation*}
for any $\delta>0$.
\end{theorem}

For any real number $\lambda>0$, define $\mu(\lambda,\alpha,\beta)=\mu(\lambda,\alpha,\beta;C,p,\CI)$ as in (\ref{defm}) in the introduction. It is easy to see that this equals $P_0(t)$, with $t=\lambda p/\abs{\CI}$. Thus by putting $k=0$ in Theorem \ref{thmPKT}, we obtain
\begin{cor}\label{corMLab}
For any $\delta>0$, under the conditions of Theorem \ref{thmPKT}, we have
\begin{equation*}
\mu(\lambda,\alpha,\beta) = e^{-\lambda(\beta-\alpha)}e^{O((1+\lambda(\beta-\alpha))/\log\log{p})} + O(p^{\frac{-1}{2}+\delta}).
\end{equation*}
Therefore, if we let $p\rightarrow\infty$, then
\begin{equation*}
\lim_{p\rightarrow\infty}\mu(\lambda,\alpha,\beta) = e^{-\lambda(\beta-\alpha)}.
\end{equation*}
Moreover, the convergence is uniform on compact subsets of $[0,\infty)$.
\end{cor}
\begin{proof}
The only thing we still need to prove is the uniform convergence on compact subsets. Unlike the primitive root case considered in \cite{CoZa98}, this comes for free, since if $p$ is large enough, \textit{every} $p$ satisfies the
conditions of Theorem \ref{thmPKT}.
\end{proof}

An important special case is obtained by letting $\CJ$ to be the full interval $[0,p)$, and $g(P,P_0)=x(P)$, the $x$-coordinates of the base point. Then if we let $\mu(\lambda)=\mu(\lambda,0,1)$, this is the proportion of consecutive $x$-coordinates in $S_{\CI}$ whose distances are greater than $\lambda p/\abs{\CI}$. We get the following direct analogue of the primitive root case considered in \cite{CoZa98}.

\begin{cor}\label{corML}
For any $\delta>0$, under the conditions of Theorem \ref{thmPKT}, and the
additional condition that $\CJ$ is the full interval $[0,p)$, then as $p$
tends to infinity,
the distribution of the number of $x$-coordinates in $S_{\CI}$ in short
intervals approaches a Poisson distribution with parameter $\lambda$:
\[\ds P_k(t) \sim e^{-\lambda}\frac{\lambda^k}{k!}.\]
Also, the distribution of the distances between consecutive $x$-coordinates satisfies
\begin{equation*}
\mu(\lambda) = e^{-\lambda}e^{O((1+\lambda)/\log\log{p})} + O(p^{\frac{-1}{2}+\delta}).
\end{equation*}
In particular, as $p\rightarrow\infty$,
\begin{equation*}
\lim_{p\rightarrow\infty}\mu(\lambda) = e^{-\lambda}.
\end{equation*}
\end{cor}

\begin{remark}
It is not absolutely necessary to consider $(x,y)\in C$ to lie in the rectangle $[0,p)\times\CI$. For example, by a linear change of variable, we can consider any parallelogram which has length $p$ (in the $x$-direction), as long as the $y$-coordinates of the rectangle lie totally inside $[0,(p-1)/2]$, and the width (in the $y$-direction) satisfies the requirement for $\abs{\CI}$.

For more general domains with piecewise smooth boundaries, one can apply the Lipschitz's principle on the number of integer points \cite{Dav51}. In that case, the error term will be much weaker, but the limiting process is still valid.
\end{remark}

\begin{remark}
It may also be interesting to see what happens if $\CI$ is too big, say $\CI=[0,(p-1)/2]$. From (\ref{estPKT}), which is still valid for this big $\CI$, we have
\begin{equation*}
P_k(t) = \binom{[t]}{k} \left( \frac{\abs{\CI}\abs{\CJ}}{p^2} \right)^k \left( 1-\frac{\abs{\CI}\abs{\CJ}}{p^2} \right)^{[t]-k} + O(p^{-1/2}\log^{2[t]+3}{p}).
\end{equation*}
As $p\rightarrow\infty$ with $\lambda=t\abs{\CI}/p$ kept constant, this just
gives
\[\ds P_k(t)\rightarrow\binom{[t]}{k} \left(\frac{1}{2}(\beta-\alpha)\right)^k\left(1-\frac{1}{2}(\beta-\alpha)\right)^{t-k}\]
and hence
\[\ds P_0(t)\rightarrow(1-1/2(\beta-\alpha))^t\sim(1-1/2(\beta-\alpha))^{2\lambda},\]
which is never close to something like $e^{-\lambda}$. The reason is that as $p\rightarrow\infty$, $t$ does not go to infinity accordingly, but stays more or less constant.
\end{remark}

\section{An application}

As an application of our results, we consider the distribution of $x$-coordinates of points on an elliptic curve in a rectangle, and the distortion of the distribution by the group law. More precisely, let $E$ be an elliptic curve defined by $y^2=x^3+ax+b$ over $\mathbb{F}_p$. Let $\CI\subset[0,(p-1)/2]$ be an interval satisfying $\abs{\CI}\leq p/\log\log{p}$. We order the points $P_i=(x_i,y_i)$ of $C$ in the rectangle $[0,p)\times\CI$ according to the size of the $x$-coordinates: $0\leq x_1 < \dots < x_m < p$.

We are interested in the distribution of the distances between consecutive $x$-coordinates $x_{i+1}-x_i$, where $x_{m+1}=x_1+p$. By Corollary \ref{corML}, the proportion of distances at least $\lambda$ times the asymptotic average $p/\abs{\CI}$ satisfies
\begin{equation*}
\lim_{p\rightarrow\infty}\frac{\#\{ i:1\leq i \leq m, x_{i+1}-x_i \geq \lambda p/\abs{\CI} \}}{m} = e^{-\lambda}.
\end{equation*}

We now look at how the group law of the elliptic curve may distort
the above distribution. Recall that (see for example \cite{Silbook})
if $P_1=(x_1,y_1), P_2=(x_2,y_2)$ are two points in $C$ with $x\neq x_0$,
the group law on $C$ reads
\begin{align}
x(P_1+P_2) &= \left(\frac{y_2-y_1}{x_2-x_1}\right)^2-x_1-x_2, \nonumber \\
y(P_1+P_2) &= -\frac{y_2-y_1}{x_2-x_1}x(P_1+P_2)-\frac{y_1x_2-y_2x_1}{x_2-x_1}, \label{gplaw} \\
-P_1 &= (x,-y). \nonumber
\end{align}

Fix an interval $\CJ=[\alpha p,\beta p)$. We want to see the proportion of
consecutive points $P_i,P_{i+1}$ (in the above sense) for which the distances between their $x$-coordinates are large, and also the $x$-coordinates of their differences $x(P_{i+1}-P_i)$ lie inside $\CJ$. From the group law (\ref{gplaw}) above, we have
\begin{equation*}
x(P_{i+1}-P_i)= \left(\frac{y_{i+1}+y_i}{x_{i+1}-x_i}\right)^2-x_i-x_{i+1}.
\end{equation*}

Suppose $P=(x,y)$ is a base point, and $P_i=(x+h_i,y_i)$. Take $g_i(P,P_i)=x(P_i-P)=\left(\frac{y_{i}+y}{h_i}\right)^2-2x-h_i$ to be the difference map. Before applying Corollary \ref{corMLab}, we need the following lemma.
\begin{lemma}
The set $\{1,g_1,\dots,g_r\}$ is linearly independent on $C_{\CH}$ for any $\CH=\{h_1,\dots,h_r\}$.
\end{lemma}
\begin{proof}
Fix any $\CH=\{h_1,\ldots, h_r\}$. Suppose there are constants $c_0, c_1, \ldots, c_r\in\mathbb{F}_p$ such that
\begin{equation*}
c_0+c_1g_1+\ldots+c_rg_r=0
\end{equation*}
on $C_{\CH}$, i.e.
\begin{equation}\label{eqnli}
c_0+c_1((h_1^{-1})^2(y_1+y)^2-2x-h_1)+\ldots+c_r((h_r^{-1})^2(y_r+y)^2-2x-h_r)=0
\end{equation}
on $C_{\CH}$. Expand this equation and notice that due to the defining equations of $C_{\CH}$, all terms with $y^2$ and $y_i^2$ can be transformed into terms involving $x$ only. This gives
\begin{equation*}
2c_1(h_1^{-1})^2yy_1+\ldots+2c_r(h_r^{-1})^2yy_r+P(x)=0,
\end{equation*}
where $P(x)$ is a polynomial in $x$ with coefficients in $\mathbb{F}_p$. Hence $c_i(h_i^{-1})^2=0$ for all $i=1,2,\ldots, r$. This implies $c_i=0$ for such $i$ since $h_i\neq 0$. By \eqref{eqnli}, we also have $c_0=0$. This completes the proof of the lemma.
\end{proof}

Now we can apply Corollary \ref{corMLab} to get
\begin{equation*}
\lim_{p\rightarrow\infty}\frac{\#\{1\leq i \leq m: x_{i+1}-x_i \geq \lambda p/\abs{\CI} \text{~and~} x(P_{i+1}-P_i)\in[\alpha p,\beta p)\}}{m} = e^{-(\beta-\alpha)\lambda}.
\end{equation*}
Thus under the extra condition about the difference map between consecutive points, the distribution of the distance $x_{i+1}-x_i$ is still of the same type, but with a different constant. Note that the new distribution only depends on the length of the interval $[\alpha p,\beta p)$, but not on the group law. That is, we get similar results for any two-point rational function $g(P,P_0)$.

\subsection*{Acknowledgements}

The authors wish to thank the referee for the suggestion of many improvements to this paper.


\end{document}